\documentclass[12pt]{amsart}
\usepackage{amssymb,latexsym,amsmath,amsthm,amscd}
\usepackage{setspace}
\usepackage[active]{srcltx}
\usepackage[all]{xy} \xyoption{arc}
\usepackage[left=2.5cm,top=2cm,right=2.5cm,bottom = 2cm]{geometry}
\usepackage{graphicx}
\usepackage[usenames,dvipsnames]{color}
\usepackage{charter}
\usepackage{hyperref}

\theoremstyle{plain}
\newtheorem{thm}{Theorem}
\newtheorem{lem}[thm]{Lemma}
\newtheorem{prop}[thm]{Proposition}

\newtheorem{conj}[thm]{Conjecture}

\theoremstyle{definition}
\newtheorem{dfn}[thm]{Definition}
\newtheorem{ex}[thm]{Example}

\theoremstyle{remark}
\newtheorem{rmk}[thm]{Remark}

\newcommand{\veps}{\varepsilon}

\newcommand*{\df}{\mathrel{\vcenter{\baselineskip0.5ex \lineskiplimit0pt
                     \hbox{\scriptsize.}\hbox{\scriptsize.}}} =}

\providecommand{\abs}[1]{\left\lvert#1\right\rvert}

\newcommand{\QQ}{\mathbf{Q}}

\newcommand{\ZZ}{\mathbf{Z}}

\newcommand{\floor}[1]{\left\lfloor#1\right\rfloor}

\begin{document}
%Title
\title[Densities of bounded primes for hypergeometric series]{Densities of bounded primes for hypergeometric series with rational parameters}
\author[Franc, Gill, Goertzen, Pas, Tu]{Cameron Franc \and Brandon Gill \and Jason Goertzen \and Jarrod Pas \and Frankie Tu}
\thanks{Corresponding author: \url{franc@math.usask.ca}. Cameron Franc was partially supported by NSERC Discovery grant RGPIN-2017-06156.}
\date{}

\begin{abstract}
  The set of primes where a hypergeomeric series with rational parameters is $p$-adically bounded is known by \cite{FrancGannonMason} to have a Dirichlet density. We establish a formula for this Dirichlet density and conjecture that it is rare for the density to be large. We prove this conjecture for hypergeometric series whose parameters have denominators equal to a prime of the form $p = 2q^r+1$, where $q$ is an odd prime, by establishing an upper bound on the density of bounded primes in this case. The general case remains open. This paper is the output of an undergraduate research course taught by the first listed author in the winter semester of 2018.
\end{abstract}
\maketitle

\tableofcontents
\setcounter{tocdepth}{2}

\section{Introduction}
The study of the coefficients of hypergeometric series has a long history. A basic question is to determine when a given series $_2F_1(a,b;c)$ with rational parameters $a$, $b$ and $c$ has integer coefficients, or perhaps more naturally, at most finitely many primes appearing in the denominators of its coefficients. This question was settled long ago by Schwarz when he classified in his famous list the series satisfying a hypergeometric differential equation with finite monodromy group. More recently, Beukers-Heckman generalized this result to series $_nF_{n-1}$ in \cite{BeukersHeckman}. In the intervening time, deep connections have been made between the arithmetic of the coefficients of $_nF_{n-1}$, quotient singularities, and the Riemann hypothesis --- see \cite{Borisov}, \cite{Bober}, \cite{RodriguezVillegas} for an introduction to this subject.

Dwork intitiated a deep study of the $p$-adic properties of hypergeometric series (see for example his books \cite{Dwork1} and \cite{Dwork2}). One question that he addressed was the $p$-adic boundedness of hypergeometric series. In \cite{Christol}, Christol gave a necessary and sufficient condition for a given hypergeometric series $_nF_{n-1}$ to be $p$-adically bounded. Recently, these studies of the congruence and integrality properties of hypergeometric series have found applications in \cite{DelaygueRivoalRoques} in the study of integrality properties of hypergeometric mirror maps. Thus, in spite of its long history, the subject of the arithmetic of hypergeometric series remains of interest.

In \cite{FrancGannonMason} a new necessary and sufficient condition for the $p$-adic boundedness of a series $_2F_1$ with rational parameters was introduced, and it was used to show that the set of bounded primes for a given series is (with finitely many exceptions) a union of primes in certain arithmetic progressions. These results have been generalized to $_nF_{n-1}$ in Tobias Bernstein's Master's thesis at the University of Alberta. The present paper continues the line of investigation opened in \cite{FrancGannonMason} toward a more global understanding of the arithmetic of hypergeometric series. Our first main result is Theorem \ref{t:density}, which reformulates the $p$-adic necessary and sufficient condition for boundedness from \cite{FrancGannonMason} into a global condition that more closely resembles the classification of hypergeometric equations with finite monodromy from \cite{BeukersHeckman} (and which goes back to work of Landau \cite{Landau} in the case of $_2F_1$). Theorem \ref{t:density} states that for all but finitely many primes, $_2F_1(a,b;c)$ is $p$-adically bounded if and only if $p$ is congruent to an element of
\[
  B(a,b;c) =\{u \in (\ZZ/m\ZZ)^\times \mid \textrm{ for all }j \in \ZZ,~\{-u^jc\} \leq \max(\{-u^ja\} ,\{-u^jb\})\}
\]
where $m$ is the least common multiple of the denominators of $a$, $b$ and $c$. 

In Section \ref{s:conjecture} we turn to the question of how the density of bounded primes behaves on average. The expectation is that it should be rare for this density to be large. For example, the Schwarz list is quite sparse among all hypergeometric series, and it is the list of series such that the density of bounded primes is equal to one. At the other end of the spectrum, \cite{FrancGannonMason} showed that $_2F_1(a,b;c)$ has a zero density of bounded primes one-third of the time. More precisely, if $a$, $b$ and $c$ are normalized to lie in the interval $(0,1)$, then the density is zero precisely when $c$ is smaller than $a$ and $b$. Conjecture \ref{mainconj} is a precise formulation of the expectation that the density of bounded primes should usually be small, and we end Section \ref{s:conjecture} with computational evidence tabulating densities of hypergeometric series with parameters of height at most $64$.

The final Section \ref{s:evidence} provides evidence that Conjecture \ref{mainconj} is true in the form of an upper bound on the densities of bounded primes of certain hypergeometric series with restricted parameters. Theorem \ref{t:specialcase} proves the following: if $p$ is a large prime of the form $p = 2q^r+1$, where $q$ is another odd prime and $r \geq 1$, then the density $D$ of bounded primes for a series
\[
  _2F_1\left(\frac{x}{p},\frac{y}{p};\frac{z}{p}\right)
\]
satisfies
\[D \leq \frac{1}{q}.\]
In particular, the density of bounded primes goes to zero if $q$ grows.

Computations suggest that $D\leq 1/q$ whenever $p \geq 59$, although we do not prove this slightly stronger result unless $p \equiv 3 \pmod{8}$. Our proof of Theorem \ref{t:specialcase} relies crucially on previously known bounds on the smallest positive nonquadratic residue mod $p$. Effective versions of our results would follow from effective versions of such upper bounds, but we do not pursue this in the present paper.

\emph{Acknowledgements}. This paper is the output of an undergraduate research course taught by Cameron Franc at the University of Saskatchewan in the winter semester of 2018. The other authors of the paper are undergraduate students that took the course. We thank Chris Soteros and Jacek Szmigielski for helping us create this course.

\section{A formula for the density of bounded primes}
\label{s:formula}
Recall that hypergeometric series $_2F_1$ are defined as power series
\[
  _2F_1(a,b;c) = \sum_{n\geq 0}\frac{(a)_n(b)_n}{(c)_n}\frac{X^n}{n!}
\]
where $(a)_n = a(a+1)\cdots (a+n-1)$ is the rising factorial. When the hypergeometric parameters $(a,b;c)$ are rational, then $_2F_1(a,b;c)$ has rational coefficients that have been the subject of considerable investigation.
\begin{dfn}
If $p$ is a prime and if $F = \sum_{n\geq 0}a_nX^n$ is a power series with $a_n \in \QQ$, then $F$ is said to be \emph{$p$-adically bounded} provided that its coefficients $a_n$ are bounded in the $p$-adic topology. Equivalently, the power of $p$ dividing the denominator of any coefficient $a_n$ is bounded from above independently of $n$. A prime $p$ is said to be \emph{bounded} for $F$ if $F$ is $p$-adically bounded.
\end{dfn}
In \cite{FrancGannonMason} it was shown that the set of bounded primes for a given $_2F_1(a,b;c)$ with rational parameters always has a Dirichlet density. Our first goal is to describe a formula for the density of bounded primes for a given hypergeometric series with rational parameters. This is achieved in Theorem \ref{t:density} below. As in \cite{FrancGannonMason}, there is little harm in assuming that the parameters $a$, $b$ and $c$ satisfy $0 < a,b,c < 1$ and $c \neq a,b$, and so we do so throughout the paper. Recall that thanks to our normalization, if $p$ is a prime such that $a-1$ is a $p$-adic unit, then $a-1$ has a perfectly periodic $p$-adic expansion of period $M$ equal to the order of $p$ in $(\ZZ/d\ZZ)^\times$, where $d$ is the denominator of $a-1$ (Lemma 2.1 of \cite{FrancGannonMason}). Let $a_j(p)$ denote the $j$th $p$-adic digit of $a-1$, and define $b_j(p)$ and $c_j(p)$ similarly. Lemma 2.3 of \cite{FrancGannonMason} showed that
\begin{equation}
\label{eq:digits}
  a_j(p) = \left \lfloor \left \{ -p^{M-1-j}a\right\}p\right \rfloor.
\end{equation}
where $\floor{x}$ denotes the floor function and $\{x\} = x - \floor{x}$. In particular,
\begin{equation}
  \label{eq:limit}
  \lim_{p\equiv u \pmod{d}} \frac{a_j(p)}{p} = \{-u^{d-1-j}a\}
\end{equation}
where the limit varies over primes $p$ within a fixed congruence class $u \pmod{d}$. Similar formulas hold for the digits of $b-1$ and $c-1$.

\begin{ex}
  Below we plot the $p$-adic digits of $-3/11$ for primes $p$ satisfying  $p\equiv 2 \pmod{11}$, which corresponds to $a = 8/11$ above. Since the order of $2$ in $(\ZZ/11\ZZ)^\times$ is $5$, the digits are periodic of period $5$. In the last column we print the digits divided by $p$ as real numbers, up to four decimal places of accuracy, to demonstrate the convergence in Equation \eqref{eq:limit}. In the last row we print the limit of the normalized digits.
  \begin{center}
    {\renewcommand{\arraystretch}{1.4}
      \begin{tabular}{c|c|c}
      $p \equiv 2\pmod{11}$ & Digits of $-3/11$ & Normalized digits\\
        \hline
        $ 13 $ & $ (8, 10, 11, 5, 9) $ & $ (0.6154, 0.7692, 0.8462,
0.3846, 0.6923) $ \\
$ 79 $ & $ (50, 64, 71, 35, 57) $ & $ (0.6329, 0.8101, 0.8987,
0.4430, 0.7215) $ \\
$ 101 $ & $ (64, 82, 91, 45, 73) $ & $ (0.6337, 0.8119, 0.9010,
0.4455, 0.7228) $ \\
$ 167 $ & $ (106, 136, 151, 75, 121) $ & $ (0.6347, 0.8144,
0.9042, 0.4491, 0.7245) $ \\
$ 211 $ & $ (134, 172, 191, 95, 153) $ & $ (0.6351, 0.8152,
0.9052, 0.4502, 0.7251) $ \\
$ 233 $ & $ (148, 190, 211, 105, 169) $ & $ (0.6352, 0.8155,
0.9056, 0.4506, 0.7253) $ \\
$ 277 $ & $ (176, 226, 251, 125, 201) $ & $ (0.6354, 0.8159,
0.9061, 0.4513, 0.7256) $\\
$ 409 $ & $ (260, 334, 371, 185, 297) $ & $ (0.6357, 0.8166,
0.9071, 0.4523, 0.7262) $\\
        \hline
        $p\to \infty$&& $(7/11, 9/11, 10/11, 5/11, 8/11)$
      \end{tabular}}
  \end{center}
\end{ex}

Recall the following result from \cite{FrancGannonMason}:
\begin{thm}
  \label{t:FGM}
  Let $a$, $b$ and $c$ denote rational numbers satisfying $0 < a,b,c < 1$ and $c\neq a,b$. Let $p$ be a prime greater than the least common multiple of the denominators of $a-1$, $b-1$ and $c-1$. Then the hypergeometric series $_2F_1(a,b;c;z)$ has $p$-adically bounded coefficients if and only if for every index $j$ we have
  \[
  c_j(p) \leq \max (a_j(p),b_j(p)).
  \]
\end{thm}
\begin{proof}
See Theorem 3.4 of \cite{FrancGannonMason}.
\end{proof}

Note that since the $p$-adic expansions of $a-1$, $b-1$ and $c-1$ are all periodic, the condition in Theorem \ref{t:FGM} only needs to be checked for a finite number of indices $j$.

\begin{thm}
\label{t:density}
  Let $a$, $b$ and $c$ denote three rational numbers satisfying $0 < a,b,c < 1$ and $c \neq a,b$. Let $m$ denote the least common multiple of the denominators of $a-1$, $b-1$ and $c-1$ when written in lowest terms, and define
\[
  B(a,b;c) = \{u \in (\ZZ/m\ZZ)^\times \mid \textrm{ for all }j \in \ZZ,~\{-u^jc\} \leq \max(\{-u^ja\} ,\{-u^jb\})\}.
\]
Then for all primes $p>m$, the series $_2F_1(a,b;c;z)$ is $p$-adically bounded if and only if $p$ is congruent to an element of $B(a,b;c)$ mod $m$. Thus, the density of the set of bounded primes for $_2F_1(a,b;c;z)$ is
\[
  D(a,b;c) = \frac{\abs{B(a,b;c)}}{\phi(m)}.
\]
\end{thm}
\begin{proof}
  Morally, this follows immediately from Theorem \ref{t:FGM} by applying the limit formula of Equation \eqref{eq:limit}. However, we want to ensure that Theorem \ref{t:density} is true exactly for the primes $p$ satisfying $p > m$, so a little more work is necessary.

  Fix an index $j$. We can write $\{-p^jc\} = C/m$, $\{-p^ja\} = A/m$ and $\{-p^jb\} = B/m$ for integers $A,B,C$ strictly betwen $0$ and $m$, where $C \neq A,B$. By Equation \eqref{eq:digits}, we must show that when $p > m$, then
  \[
    \floor{ Cp/m} \leq \max\left(\floor{ Ap/m},\floor{Bp/m}\right)
  \]
  if and only if
  \[
C/m \leq \max(A/m ,B/m).
\]
That the latter implies the former is obvious (multiply by $p$ and take the floor). Conversely, suppose without loss of generality that $A/m \geq B/m$ and $\floor{ Cp/m} \leq \floor{ Ap/m}$. If $\floor{ Cp/m} < \floor{ Ap/m}$ then obviously $C/m < A/m$ and we're done. Hence assume that $\floor{Cp/m} = \floor{Ap/m}= N$, so that
\[
  \frac{Nm}{p} \leq A,C < \frac{Nm}{p}+\frac{m}{p}.
\]
Since $p > m$ and $A$ and $C$ are integers, it follows from this that $A = C$, a contradiction.
\end{proof}
Theorem \ref{t:density} has several useful consequences. For example, observe that $B(a,b;c)$ is a union of cyclic subgroups of $(\ZZ/m\ZZ)^\times$. In particular, $B(a,b;c) = \emptyset$ if and only if $1 \not \in B(a,b;c)$. From this one deduces that $D(a,b;c) = 0$ if and only if $c <a$ and $c < b$. This characterization of when the density of bounded primes satisfies $D(a,b;c) = 0$ was first established in Theorem 4.14 of \cite{FrancGannonMason} without the use of Theorem \ref{t:density} above.

\section{A conjecture on the densities of bounded primes}
\label{s:conjecture}
In this section we consider the general behaviour of the density $D(a,b;c)$ of bounded primes for a hypergeometric series. The expectation is that it should be rare for this density to be large. For example, in \cite{FrancGannonMason} it was shown that $D(a,b;c) = 1$ if and only if the corresponding monodromy representation is finite\footnote{That finite monodromy implies $D(a,b;c) = 1$ is a result that goes back to Eisenstein}. In order to study this question we introduce a notion of complexity for the parameters of a hypergeomteric series. Recall that if $a \in \QQ$, then the \emph{height} $h(a)$ of $a$ is the maximum size of the numerator or denominator of $\abs{a}$ when $\abs{a}$ is written in lowest terms.

\begin{dfn}
The \emph{parameter height} of a hypergeometric series $_2F_1(a,b;c)$ is the quantity
\[
  h(a,b;c) \df \max \{h(a),h(b),h(c)\}.
\]
\end{dfn}

The parameter height of $_2F_1(a,b;c)$ has nothing to do with the usual height of the rational coefficients of $_2F_1(a,b;c)$.  Since we're normalizing our parameters to lie in the interval $(0,1)$, the parameter height is determined by the denominators of the parameters $a$, $b$ and $c$. 

Note that if $a$, $b$ and $c$ are any rational numbers such that $a$, $b$, $c$, $a-c$ and $b-c$ are not integers, then
\[D(a,b;c) = D(\{a\},\{b\};\{c\}),\]
although a finite number of bounded primes for one of the sets of parameters above could be unbounded for the other (see \cite{FrancGannonMason}). We thus let $P$ denote the parameter set of hypergeometric triples $(a,b;c)$ where $a,b,c \in (0,1) \cap \QQ$ and $c\neq a,b$.

For $r \in [0,1]$ define
\[
  \beta(r,N) = \frac{\abs{\{(a,b;c) \in P \mid h(a,b;c) \leq N \textrm{ and } D(a,b;c) \leq r\}}}{\abs{\{(a,b;c) \in X \mid h(a,b;c) \leq N\}}}.
\]
Then $\beta(r,N)$ measures what proportion of hypergeometric series have a density of bounded primes that is at most $r$. For example, \cite{FrancGannonMason} showed that, under our hypotheses on $a$, $b$ and $c$, one has $D(a,b;c) =0$ exactly when $c$ is the smallest parameter. Since $h(a,b;c)$ is invariant under permuting $a$, $b$ and $c$, it follows that
\[
  \beta(0,N) = \frac{1}{3}.
\]
Computations suggest that if $\veps > 0$, then for large enough $N$, the proportion $\beta(\veps,N)$ of hypergeometric series with a density of bounded primes that is at most $\veps$ should be quite large. In fact, the following conjecture is supported by computational evidence:
\begin{conj}
  \label{mainconj}
For all $\veps > 0$,
\[
  \lim_{N \to \infty} \beta(\veps,N) = 1.
\]
\end{conj}

We have performed extensive computations of densities of bounded primes for all hypergeometric series with parameters normalized as above, and satisfying $h(a,b;c) \leq 64$. First we plot the frequency of each density count up to height $16$ in Figure \ref{figure1}. Observe that the frequency of density zero in Figure \ref{figure1} accounts for one-third of the data, and it dominates the figure. Thus, in Figure \ref{figure2} we include similar plots up to heights $16$, $32$, $48$ and $64$, but we omit the data for density zero.

\begin{figure}[h]

  \includegraphics[scale=0.6]{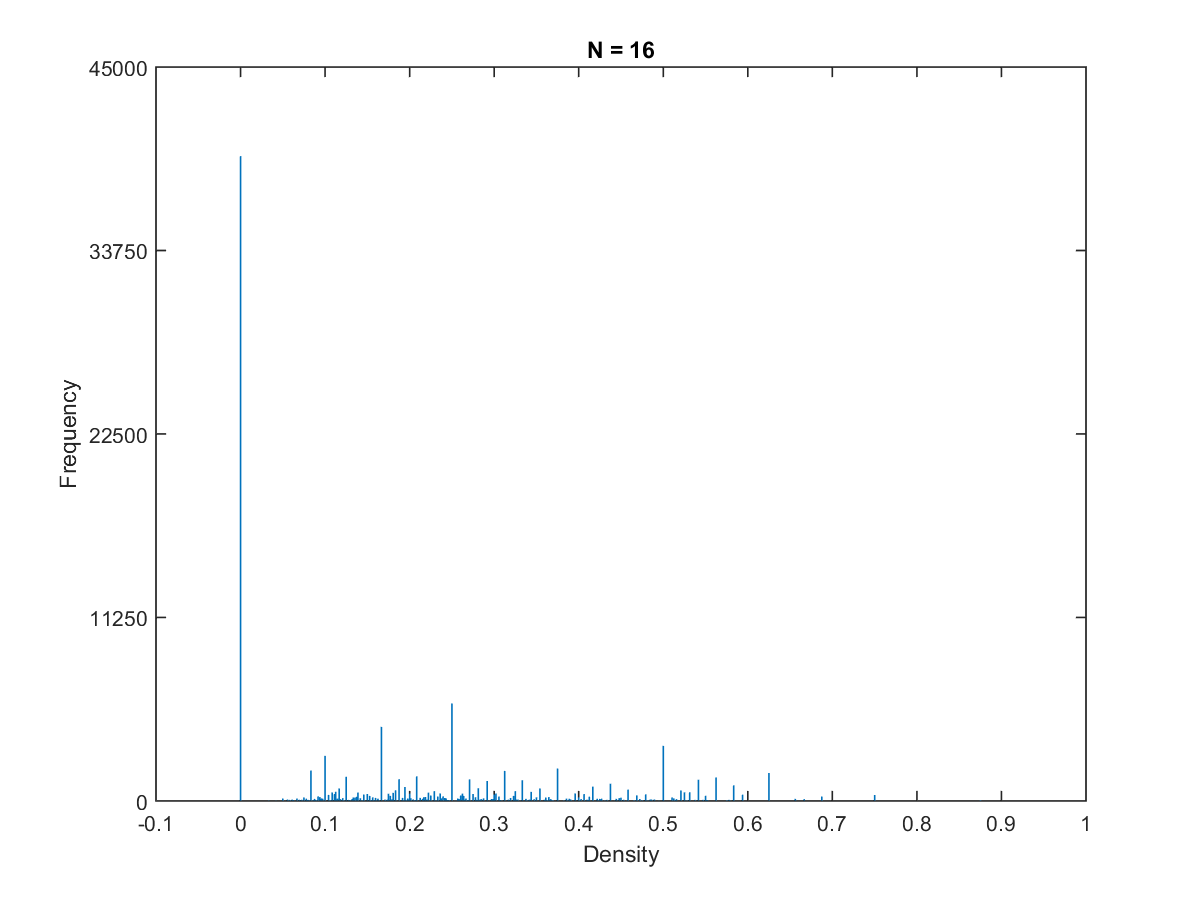}
  \caption{Densities of bounded primes for hypergeometric series up to parameter height $16$.}
  \label{figure1}
\end{figure}

\begin{figure}[h]

  \hspace{-4.5cm}\begin{minipage}{\textwidth}
  \begin{tabular}{cc}
\includegraphics[scale=0.5]{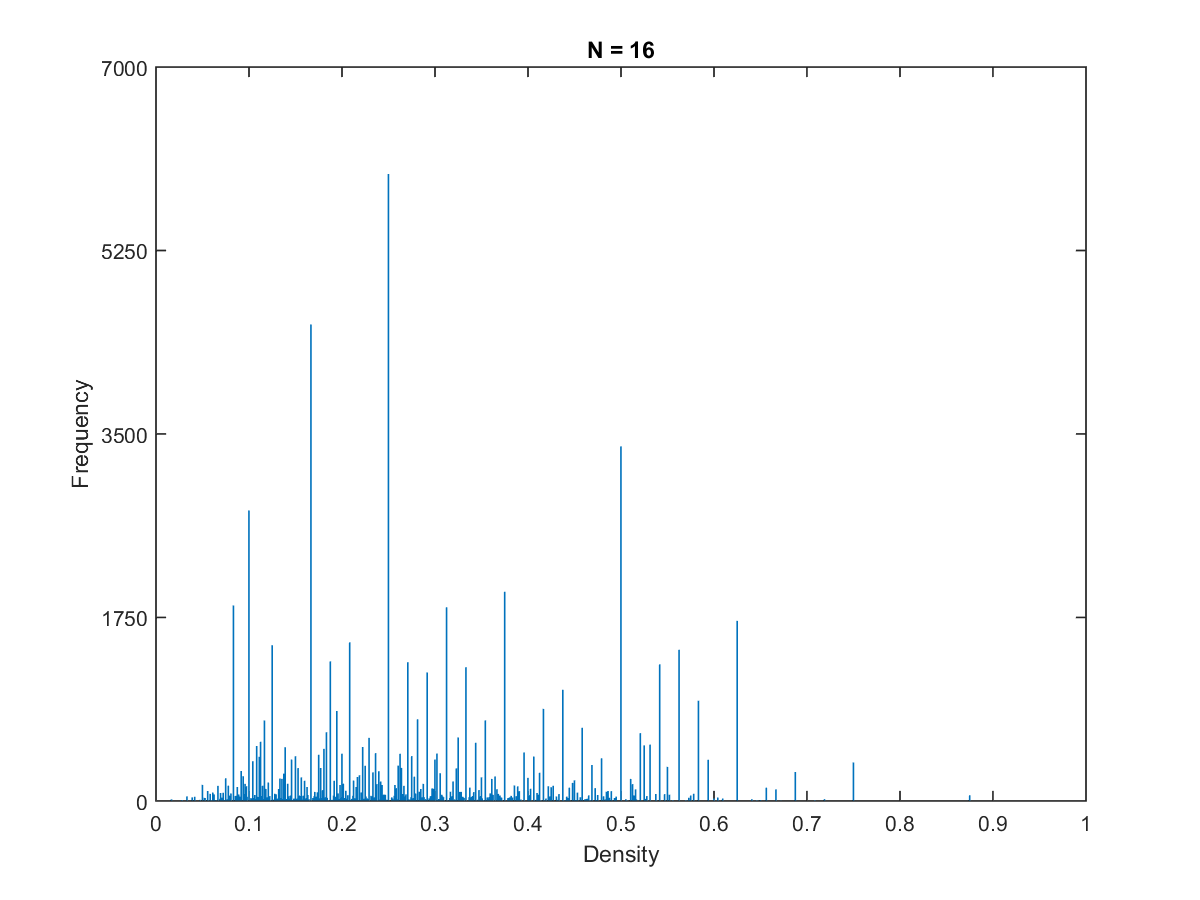}  &
\includegraphics[scale=0.5]{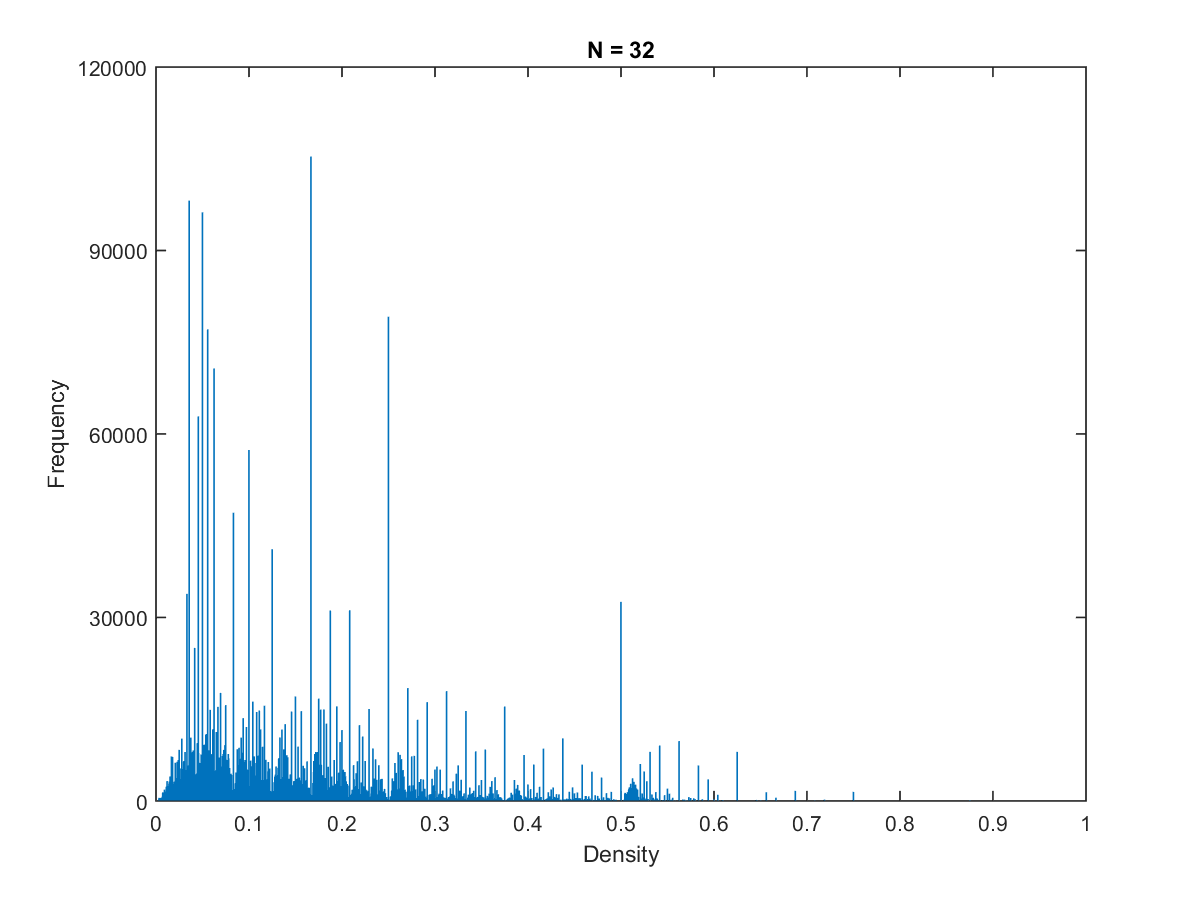}  \\
\includegraphics[scale=0.5]{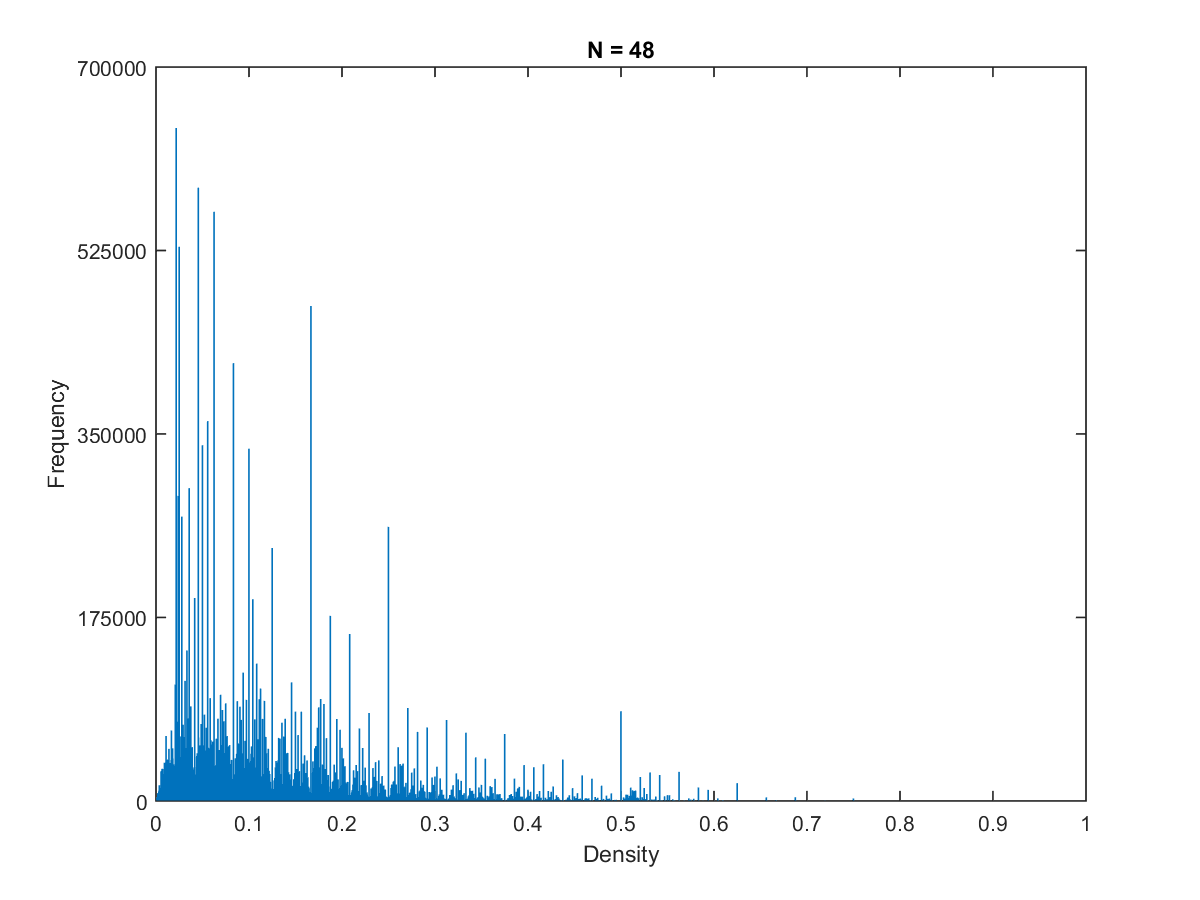}  &
\includegraphics[scale=0.5]{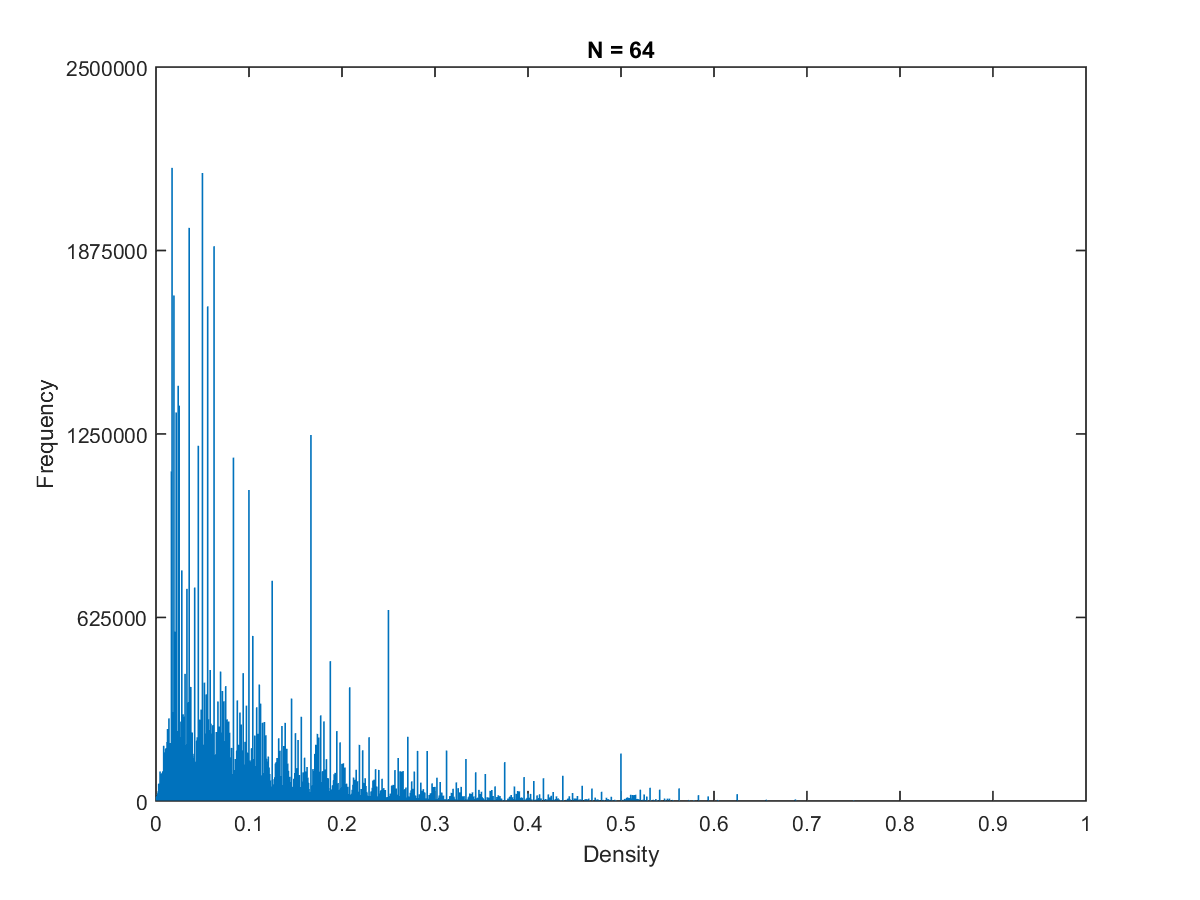} 
\end{tabular}
\end{minipage}
\caption{All densities of bounded primes up to parameter height $N = 16$, $32$, $48$ and $64$, with the zero density counts removed.}
\label{figure2}
\end{figure}

These plots contain spikes at certain densities, such as those of the form $1/2^n$, but the trend appears to be that the densities concentrate towards zero as the height grows. This is consistent with Conjecture \ref{mainconj}. In the next section we prove Theorem \ref{t:specialcase} which bounds certain densities of bounded primes from above, and which provides  more evidence that supports the truth of Conjecture \ref{mainconj}.

\begin{rmk}
The classical Schwarz list classifies those hypergeometric parameters for which $D(a,b;c) = 1$. Given $\delta > 0$, one might similarly try to classify all those parameters such that $D(a,b;c) \geq \delta$ in some kind of $\delta$-Schwarz list. If Conjecture \ref{mainconj} is true, then such a classification might not be unreasonable for certain values of $\delta$, such as $\delta = 1/2^n$.
\end{rmk}

\section{An upper bound for densities of bounded primes}
\label{s:evidence}
Let $m$ denote the least common multiple of the denominators of rational hypergeometric parameters $a$, $b$ and $c$, and assume that $U_m$ is cyclic. Let $u \in U_m$ be a generator. Then for every divisor $d \mid \phi(m)$, the element $u_d = u^{(p-1)/d}$ has order $d$. The possibilities for $B(a,b;c)$ are the sets generated by the various $u_d$:
\[
 \langle u_d \rangle = \{u_d,u_d^2,\ldots, u_d^{d}\}.
\]
If $x$ is a positive integer then let $d(x)$ denote the set of positive divisors of $x$. Let $I_x =\{J \subseteq d(x) \mid d \nmid e \textrm{ for all } d,e \in J \}$. Then the possibilities for $B = B(a,b;c)$ are the sets
\[
  B_J = \bigcup_{d \in J} \langle u_d\rangle
\]
for all $J \in I_{\phi(m)}$. Note that the unions are definitely not disjoint. By the inclusion-exclusion principle,
\[
  \abs{B_J} = \sum_{\substack{K \subseteq J\\ K \neq \emptyset}}(-1)^{\abs{K}-1}\gcd(K)
\]
where $\gcd(K)$ denotes the greatest common divisor of the elements of $K$. 

We will work with a choice of $m$ so that the subgroup structure of $U_m$ is very simple: let $m$ be a prime of the form $p= 2q^r+1$ where $q$ is also an odd prime, and $r \geq 1$. Note that then $p \equiv 3 \pmod{4}$ and $p > 3$. If $u$ is a generator for $U_p$ then the discussion above specializes to the following possibilities for the subsets $B(a,b;c)$ and the corresponding densities of bounded primes:
\begin{center}
\begin{tabular}{|l|l|}
\hline
$B(a,b;c)$ & Density \\
\hline
$\emptyset$  & 0\\
$\langle u^{q^j}\rangle$ &  $1/(2q^j)$\\
$\langle -u^{q^k}\rangle$ & $1/q^k$\\
$\langle u^{q^j}\rangle \cup \langle -u^{q^k}\rangle$ ($j < k$) & $(q^{k-j}+1)/(2q^{k})$\\
\hline
\end{tabular}
\end{center}
Note that most of these densities are quite small. In Theorem \ref{t:specialcase} below we will show that if $p$ is very large, then the large densities never occur. 

For the moment let $p > 3$ be any odd prime such that $p \equiv 3 \pmod{4}$, so that $-1$ is not a quadratic residue mod $p$. Let $Q$ denote the set of quadratic residues in $(\ZZ/p\ZZ)^\times$. Let $h_p$ denote the class number of $\QQ(\sqrt{-p})$. Since $p \equiv 3\pmod{4}$ and $p > 3$, Dirichlet's analytic class number formula implies that
\begin{equation}
\label{eq:dirichlet}
  -ph_p = \sum_{0 < y < p} \chi(y)y
\end{equation}
where $\chi(y) = \left(\frac{y}{p}\right)$ is the Legendre symbol.

If $x$ is an integer or an element of $\ZZ/p\ZZ$, then let $[x]_p$ denote the unique integer congruent to $x$ mod $p$ such that $0 \leq [x]_p < p$. Define
\begin{align*}
  U_p(x) &= \{y \in (\ZZ/p\ZZ)^\times \mid [xy]_p < [y]_p\},\\
  W_p(x) &=  U_p(x) \cap Q,
\end{align*}
and let $w_p(x) = \abs{W_p(x)}$.

Below we will see that the proof of Theorem \ref{t:specialcase}, which gives an upper bound on the densities of bounded primes considered here, follows from the fact that the intersection $W_p(a/c) \cap W_p(b/c)$ is nonempty whenever $p$ is large enough. In Proposition \ref{p:classnumber} below we show that the sets $W_p(x)$ are relatively large and often must intersect for trivial reasons. However, this simple proof does not work in all cases, and so more work is required. Proposition \ref{p:intersection} establishes what we need on the nonemptiness of these intersections, but it will require a sequence of preliminary results.
\begin{lem}
  \label{l:Uproperties}
  Let $p$ be an odd prime. If $x \not \equiv 0,1\pmod{p}$, the sets $U_p(x)$ satisfy the following:
  \begin{enumerate}
  \item $U_p(1)$ is empty;
  \item $U_p(x)$ contains one element of each pair $\{y,-y\}$ and thus $\abs{U_p(x)}=\frac{p-1}{2}$;
  \item $U_p(x) = U_p(1-x)$;
   \item $(x-1)U_p(x) = U_p(x/(x-1))$.
   \end{enumerate}
   Finally, if $1 \leq x \leq p-2$ then the set $U_p(-x)$ breaks up into disjoint intervals as follows:
  \[
U_p(-x) = \bigcup_{a=1}^x \left\{\frac{ap}{x+1} < y < \frac{ap}{x}\right\}.
\]
\end{lem}
\begin{proof}
  Part (1) is clear, and (2) follows since $[-y]_p = p-[y]_p$.

  For (3) observe that $[xy]_p < [y]_p$ implies that we can write $xy = ap+r$ for $0 \leq r < y < p$. But then $(1-x)y = -ap+(y-r)$ shows that $[(1-x)y]=y-r < [y]$. Hence $U(x) \subseteq U(1-x)$, and the reverse inclusion follows by replacing $x$ with $1-x$.

For part (4), define
  \[
  V_p(x) = \{y \in (\ZZ/p\ZZ)^\times \mid [y]_p < [xy]_p\}.
\]
Observe that as above, $V_p(x) = V_p(1-x)$ for $x \not \equiv 0,1\pmod{p}$. Since $xU_p(x) = V_p(x^{-1})$ we find that
\[
  xU_p(x) = V_p(x^{-1}) = V_p(1-x^{-1}) = V_p((x-1)/x) = (x/(x-1))U_p(x/(x-1)).
\]
This establishes (4).

  Finally, for the last statement, observe that if $\frac{ap}{x+1} < y < \frac{ap}{x}$, this is equivalent with $0 < ap-xy < y$. It follows that $[-xy]_P < y =[y]_p$, and hence $y \in U_p(-x)$.

  Conversely, if $[-xy]_p < [y]_p$ we may write $ap-xy = [-xy]_p$. We are free to assume $0 \leq y < p$, so that $[y]_p = y$ and $a > 0$. Similarly $ap-xy < y$ so that $ap < (x+1)y < (x+1)p$. Hence $1 \leq a \leq x+1$ and we find that $\frac{ap}{x+1} < y < \frac{ap}{x}$. This implies that $U_p(-x)$ is as described.
\end{proof}

The following Proposition is not strictly necessary for the proof of Theorem \ref{t:specialcase}, but we include it out of interest. It demonstrates that the while the sets $W_p(u)$ and $W_p(v)$ sometimes must intersect for trivial reasons (say if $u$, $1-u$, $v$ and $1-v$ are all quadratic residues), such a simple-minded argument does not work in all cases.
\begin{prop}
  \label{p:classnumber}
Let $p$ be a prime satisfying $p > 3$ and  $p \equiv 3\pmod{4}$, and write $p = 2n+1$. Let $\chi(x) = \left(\frac xp\right)$ denote the Legendre symbol. If $x \not \equiv 0,1 \pmod{p}$ then
\[
	w_p(x) = \frac{1}{2}\left(n + (\chi(x)+\chi(1-x)-1)h_p\right) = \begin{cases}
	\frac{n+h_p}{2} & \chi(x) = \chi(1-x) = 1,\\
	\frac{n-h_p}{2} & \chi(x) \neq \chi(1-x),\\
	\frac{n-3h_p}{2} & \chi(x)=\chi(1-x)=-1.
	\end{cases}
\]
\end{prop}
\begin{proof}
Begin by writing
\[
  w(x) = \sum_{y \in W_p(x)} 1 = \frac 12\sum_{y \in U_p(x)}(1+\chi(y)) = \frac{n}{2} + \frac 12 \sum_{y \in U_p(x)}\chi(y).
\]
Therefore we define $u(x) = \sum_{y \in U_p(x)}\chi(y)$. We must show that
\[
  u(x) = (\chi(x)+\chi(1-x)-1)h_p.
\]
Since $U_p(x)$ contains one of each pair $\{y,-y\}$ of elements,
\[
  \sum_{0 < y < p}\chi(y)y = \sum_{y \in U_p(x)}(\chi(y)y + \chi(-y)(p-y)) = 2\sum_{y\in U_p(x)}\chi(y)y - pu(x). 
\]
By part (4) of Lemma \ref{l:Uproperties}, we know that $(x-1)U_p(x) = U_p(x/(x-1))$. Therefore,
\begin{align*}
  \sum_{0 < y < p}\chi(y)y &=\sum_{y \in U_p(x/(x-1))}(\chi(y)y+\chi(-y)(p-y))\\
                           &=\sum_{y \in (x-1)U_p(x)}(\chi(y)y+\chi(-y)(p-y))\\
                           &=\sum_{w \in U_p(x)}(\chi((x-1)w)[(x-1)w]_p+\chi(-(x-1)w)[-(x-1)w])_p\\
  &= \chi(x-1)\sum_{w \in U_p(x)}\chi(w)([(x-1)w]_p-[-(x-1)w]_p)
\end{align*}
Note that $[-(x-1)w]_p = p-[(x-1)w]_p$. Hence,
\[
 \sum_{0 < y < p}\chi(y)y = 2\chi(x-1)\sum_{y \in U_p(x)}\chi(w)[(x-1)w]_p - p\chi(x-1)u(x).
\]
Using that $U_p(x) = U_p(1-x)$ by Lemma \ref{l:Uproperties}, we can replace $x$ by $1-x$ above to deduce that
\[
 \sum_{0 < y < p}\chi(y)y = 2\chi(-x)\sum_{y \in U_p(x)}\chi(w)[-xw]_p - p\chi(-x)u(x).
\]
If we now apply Dirichlet's analytic class number formula \eqref{eq:dirichlet} to these three identities, we've shown that
\begin{align*}
  -ph_p &= -pu(x)+2\sum_{w\in U_p(x)}\chi(w)w\\
  \chi(1-x)ph_p &= -pu(x)+2\sum_{w\in U_p(x)}\chi(w)[(x-1)w]_p\\
  \chi(x)ph_p &=-pu(x)+ 2\sum_{w \in U_p(x)}\chi(w)[-xw]
\end{align*}
Therefore,
\[
  (\chi(x)+\chi(1-x)-1)ph_p = -3pu(x) + 2\sum_{w \in U_p(x)}\chi(w)(w+[(x-1)w]_p+[-xw]_p)
\]
Hence what we're trying to prove is equivalent to showing that
\[
 u(x) = \frac{1}{2p}\sum_{w \in U_(x)}\chi(w)(w+[(x-1)w]_p+[-xw]_p)
\]
Observe that
\begin{align*}
  w+[(x-1)w]_p+[-xw]_p &= 2p+w-[(1-x)w]_p-[xw]_p.
\end{align*}
Since $w \in U_p(x)$, we have $[xw]_p < [w]_p=w$. Write $xw = ap+r$ where $0 \leq r < w$. Then $(1-x)w = w-ap-r = -ap+(w-r)$ shows that $[(1-x)w]_p = w-r$. Therefore
\[
  \frac{1}{2p}\sum_{w \in U(x)}\chi(w)(w+[(x-1)w]_p+[-xw]_p) = \frac{1}{2p}\sum_{w \in U(x)}\chi(w)(2p) = u(x), 
\]
as desired.
\end{proof}

\begin{rmk}
  Part (5) of Lemma \ref{l:Uproperties} shows that Proposition \ref{p:classnumber} is equivalent to the following formula: if $1 \leq x \leq p-2$ and $p$ is a prime satisfying $p \equiv 3 \pmod{4}$, then 
  \[
  \sum_{a=1}^x\sum_{\frac{ap}{x+1} < y < \frac{ap}{x}} \left(\frac{y}{p}\right) = (\chi(x+1)-\chi(x)-1)h_p.
\]
There is a long history concerning formulas for sums of Legendre symbols over restricted intervals in terms of class numbers (see for example the paper \cite{Berndt} or the more recent text \cite{Wright}), although we were not able to find this particular result in the literature.
\end{rmk}

Let $p$ denote an odd prime, and let $n_p > 1$ denote the smallest integer that is not a quadratic residue mod $p$. The Riemann hypothesis for Dirichlet $L$-series implies that
\[
  n_p < \frac 32 (\log p)^2.
\]
The following weaker but unconditional result is the best known bound to date.
\begin{thm}
  \label{t:smallnp}
  For all odd primes $p$ one has
  \[n_p = O_\veps(p^{\frac{1}{4\sqrt{e}}+\veps}).\]
\end{thm}
\begin{proof}
See \cite{Burgess}.
\end{proof}

\begin{lem}
  \label{l:Uproperties2}
  Let $p$ be an odd prime. The following statements hold.

  \begin{enumerate}
\item If $y \in U_p(x)$ and $0 < y < p$, then $yj \in U_p(x)$ for all $j$ in the range $1 \leq j \leq \floor{p/y}$.
  \item Let $0 < r < p$ and $n = p-r$, and suppose $r \not \in U_p(x)$. If $N = \floor{p/n}$, and if
  \[
  Y = \{nj \mid 1 \leq j \leq N\},
\]
then $Y \subseteq U_p(x)$.
\end{enumerate}
\end{lem}
\begin{proof}
  For the first claim, since $y \in U_p(x)$, we can write $xy = ap + r$ for $0 \leq r < y < p$. Hence $x(yj) = ajp + rj$ where $0 \leq rj < yj < p$. Then
  \[[x(yj)]_p = rj < yj = [yj]_p\]
 shows that $yj \in U_p(x)$.

 For the second claim observe that if $r \not \in U_p(x)$ then $n \in U_p(x)$ by part (2) of Lemma \ref{l:Uproperties}. Thus (2) follows from (1).
\end{proof}

\begin{ex}
  \label{ex:threemodeight}
  If $p \equiv 3 \pmod{8}$ then $n_p = 2$ is the smallest positive nonquadratic residue mod $p$ and $r_p = p-2$ is the largest quadratic residue between $0$ and $p$. It's not hard to show that $r_p \not \in U_p(x)$ if and only if $x \equiv (p+1)/2 \pmod{p}$. In this case Lemma \ref{l:Uproperties2} gives a complete description of $U_p((p+1)/2)$ as
  \[
  U_p((p+1)/2) = \{2j \mid 1 \leq j \leq (p-1)/2\}.
\]
Hence $W_p((p+1)/2)$ is the set of quadratic residues $y$ whose least positive residue $[y]_p$ is even. By Proposition \ref{p:classnumber}, there are $\frac{1}{2}(\frac{p-1}{2} - 3h_p)$ such quadratic residues. This classical result is well-known.
\end{ex}

More generally, Theorem \ref{t:smallnp} and Lemma \ref{l:Uproperties2} together imply that if $p \equiv 3 \pmod{4}$ and $U_p(x)$ does not contain the largest quadratic residue $r_p$ in the range $0 < r_p < p$, then there is a large subset $Y$ of $U_p(x)$ that we understand well. We will use this subset $Y$ to prove that the various sets $W_p(x)$ must intersect if $p$ is large enough.

\begin{prop}
  \label{p:intersection}
  Let $p \equiv 3 \pmod{4}$ be prime. Then there exists an integer $N$ such that if $p > N$, then for all $u,v$ coprime to $p$ and satisfying $u \not \equiv 1\pmod{p}$, $v \not \equiv 1 \pmod{p}$, we have
  \[
  W_p(u) \cap W_p(v) \neq \emptyset.
  \]
\end{prop}
\begin{proof}
  Let $n_p$ denote the smallest positive nonquadratic residue, and let $r_p = p-n_p$ denote the largest quadratic residue in the range $0 < r_p < p$. There is nothing to prove if $r_p\in W_p(u) \cap W_p(v)$. Suppose instead that $r_p \not \in W_p(u)$ and $r_p \not \in W_p(v)$. In this case Lemma \ref{l:Uproperties2} yields that
\[
  \{n_pj \mid 1 \leq j \leq \floor{p/n_p}\} \subseteq U_p(u) \cap U_p(v).
\]
Hence, we must show that the interval $1 \leq j \leq \floor{p/n_p}$ contains $n_p$. This follows by Theorem \ref{t:smallnp} for large enough primes $p$, although weaker bounds would work for this case.

  Finally suppose that $r_p \in W_p(u)$ but $r_p \not \in W_p(v)$. Since $r_p = p-n_p$ satisfies $r_p\not \in U_p(v)$, Lemma \ref{l:Uproperties2} tells us that with $N= \floor{p/n_p}$, we have the containment
\[
  \{n_pj \mid 1 \leq j \leq N\}\subseteq U_p(v).
\]
If $U_p(u)$ contains a nonquadratic residue $j$ in the range $1 \leq j \leq N$, then it also contains the quadratic residue $n_pj$ by Lemma \ref{l:Uproperties2}. Then $n_pj \in W_p(u) \cap W_p(v)$. So we may assume that $U_p(u)$ does not contain any nonquadratic residue in the range $1\leq  j \leq N$. Equivalently, by part (2) of Lemma \ref{l:Uproperties}, $U_p(u)$ contains every quadratic residue in the range $p-N \leq y \leq p-1$.

Suppose that $n_pj \geq p-N$ for some $1 \leq j \leq N$. This means that
\begin{equation}
\label{mainineq}
  \frac{n_p-1}{n^2_p}p < j < \frac{p}{n_p}.
\end{equation}
Observe that $W_p(u) \cap W_p(v)$ contains all the elements $n_pj$ for nonquadratic residues $j$ satisfying the inequality \eqref{mainineq}. If there are no nonquadratic residues $j$ in this range, then the Polya-Vinogradov inequality
\[
  \abs{\sum_{a=b}^{N}\left(\frac{a}{p}\right)} < \sqrt{p}\log(p)
\]
with $b = \lceil (n_p-1)p/n_p^2\rceil$ gives
\[
  (p/n^2_p)-1 <\sqrt{p}\log p.
\]
But then Theorem \ref{t:smallnp} implies that
\[
  p^{1-\frac{1}{2\sqrt{e}}-2\veps} = O_\veps(\sqrt{p}\log(p)).
\]
This is a contradiction for small enough $\veps$, since $\frac{1}{2} - \frac{1}{2\sqrt{e}} > 0$. Therefore, for large enough primes, there is a nonquadratic residue $j$ in the range \eqref{mainineq}, and then $n_pj \in W_p(u) \cap W_p(v)$. This concludes the proof.
\end{proof}

If $p\equiv 3 \pmod{8}$ then we can use the fact that $2$ is the smallest positive nonquadratic residue mod $p$ to get an effective version of Proposition \ref{p:intersection}.
\begin{prop}
\label{p:effectiveintersection}
Let $p \equiv 3 \pmod{8}$ be prime. If $p > 11$, then for all $u,v$ coprime to $p$ and satisfying $u \not \equiv 1\pmod{p}$, $v \not \equiv 1 \pmod{p}$, we have
  \[
  W_p(u) \cap W_p(v) \neq \emptyset.
  \]
\end{prop}
\begin{proof}
  In this case $n_p = 2$ and $r_p = p-2$. Again, there is nothing to show if $r_p \in W_p(u) \cap W_p(v)$. Suppose that $r_p \not \in W_p(u)$ and $r_p \not \in W_p(v)$. As in the proof of Proposition \ref{p:intersection}, we must prove that the interval $1 \leq j \leq \floor{p/2}$ contains $n_p$. This is obvious if $p > 3$.
  
  Finally assume that $r_p \in W_p(u)$ but $r_p \not \in W_p(v)$. Set $p=2q+1$, so that by Example \ref{ex:threemodeight},
\[
  W_p(q+1) = \{2,4,\ldots, p-1\}\cap Q.
\]
We must show that every set $W_p(u)$ contains a quadratic residue $y$ such that $[y]_p$ is an even integer.

If $U_p(u)$ contains a nonquadratic residue in the range $0 < y < p/4$ then by part (1) of Lemma \ref{l:Uproperties2}, $U_p(u)$ also contains the quadratic residue $2y = [2y]_p$, and hence $2y \in W_p(u)$ is an even element. Similarly, if $U_p(u)$ contains a quadratic residue in this range, then it also contains $4y = [4y]_p$ and $W_p(u)$ contains the even element $4y$. Therefore we can assume that $U_p(u)$ does not contain any element in the range $0 < y < p/4$. This means that $U_p(u)$ contains every element in the range $3p/4 < y < p$, by part (2) of Lemma \ref{l:Uproperties}.

Conisder the elements $3p/8 < y < p/2$. If there is a nonquadratic residue in this range then $3p/4 < 2y < p$ and $2y$ is an even quadratic residue, and it will thus be contained in $W_p(u)$. So we are reduced to proving that if $p$ is large enough, then there is always a nonquadratic residue in the range $3p/8 < y <p/2$. If we write $p = 8k+3$ then $\lceil 3p/8\rceil = 3k+1$ and $\floor{p/2} = 4k$, and if no nonquadratic residue is in this range, then the Polya-Vinogradov inequality gives
\[
  k = \abs{\sum_{a=3k+1}^{4k}\left(\frac{a}{p}\right)} < \sqrt{p}\log p.
\]
That is, we obtain $\frac{p-3}{8} < \sqrt{p}\log p$, which is a contradiction if $p > 568$. The remaining cases for primes $p > 11$ can be checked easily on a computer.
\end{proof}

\begin{ex}
  The bound of $p > 11$ in Proposition \ref{p:effectiveintersection} cannot be improved. For example,
  \begin{align*}
    W_{11}(2)& = \{9\}, & W_{11}(6) &= \{4\}, 
  \end{align*}
and $W_{11}(2) \cap W_{11}(6) = \emptyset$.
\end{ex}

\begin{thm}
\label{t:specialcase}
There exists a number $N$ such that the following holds: for all primes $p > N$ of the form $p = 2q^r+1$, with $q$ an odd prime and $r \geq 1$, and for all rational parameters $x/p$, $y/p$ and $z/p$ where $1 \leq x,y,z \leq p-1$  and $x \neq z$, $y \neq z$, the density of bounded primes for
\[
  _2F_1\left(\frac{x}{p},\frac{y}p;\frac zp\right)
\] is bounded above as follows:
\[
  D\left(\frac xp,\frac yp;\frac zp\right) \leq \frac{1}{q}.
\]
In particular, if $q$ grows, then the density of bounded primes goes to zero.
\end{thm}
\begin{proof}
  Observe that if $p = 2q^r+1$ is prime for an odd prime $q$, then $p \equiv 3 \pmod{4}$ and so $-1$ is not a quadratic residue mod $p$. Let $u \in (\ZZ/p\ZZ)^\times$ have order $q^r$, so that $u$ generates the group $Q$ of quadratic residues mod $p$. Let $B = B(x/p,y/p;z/p)$ denote the set of congruence classes mod $p$ representing bounded primes for our parameters, as in Theorem \ref{t:density}. Since $B$ is a union of cyclic subgroups, there are a limited number of possibilities for the form it can take, and for the corresponding densities of bounded primes. The cyclic subgroups of $(\ZZ/p\ZZ)^\times$ have the form $\langle u^{q^j}\rangle$ for $0 \leq j \leq r$ or $\langle -u^{q^j}\rangle$ for $0 \leq j \leq r$. Besides the obvious containments, we have $\langle u^{q^j}\rangle \subseteq \langle -u^{q^k}\rangle$ if and only if $k \leq j$. If $j < k$ then $\langle u^{q^j}\rangle \cap \langle -u^{q^k}\rangle = \langle u^{q^k}\rangle$ and
\[
  \abs{\langle u^{q^j}\rangle \cup \langle -u^{q^k}\rangle} = \abs{\langle u^{q^j}\rangle}+\abs{\langle -u^{q^k}\rangle}-\abs{\langle u^{q^k}\rangle} = q^{r-j}+2q^{r-k}-q^{r-k} = q^{r-j}+q^{r-k}.
\]
Therefore, we have the following possibilities for $B$ and the corresponding density of bounded primes:
\begin{center}
\begin{tabular}{|l|l|}
\hline
$B(a,b;c)$ & Density \\
\hline
$\emptyset$  & 0\\
$\langle u^{q^j}\rangle$ &  $1/2q^j$\\
$\langle -u^{q^k}\rangle$ & $1/q^k$\\
$\langle u^{q^j}\rangle \cup \langle -u^{q^k}\rangle$ ($\j < k$) & $(q^{k-j}+1)/(2q^{k})$\\
\hline
\end{tabular}
\end{center}
The problematically large densities satisfying
\[
  D(a,b;c) > \frac{1}{q}
\]
are the cases when $j = 0$ and $k = 0$ in the second and third rows, and the cases when $j = 0$ in the last row. That is, the problematic cases are:
\begin{center}
\begin{tabular}{|l|l|}
\hline
$B(a,b;c)$ & Density \\
\hline
$\langle u\rangle$ &  $1/2$\\
$\langle -u\rangle$ & $1$\\
$\langle u\rangle \cup \langle -u^{q^k}\rangle$ ($k > 0$) & $(q^{k}+1)/(2q^{k})$\\
\hline
\end{tabular}
\end{center}
Observe that the problematic cases are exactly those for which $u \in B$. Thus, we must produce a value of $N$ such that if $p > N$ then $u \not \in B$.

Theorem \ref{t:density} shows that $u \in B$ if and only if
\[\{-u^jz/p\} \leq \max(\{-u^jx/p\} ,\{-u^jy/p\})\]
for all integers $j$. Notice that $ \{-u^jz/p\}p = [-u^jz]_p$, and likewise for the other terms appearing above. As $j$ varies, the term $u^j$ varies over the set $Q$ of quadratic residues mod $p$. Therefore, we find that $u \in B$ if and only if
\[
  [-vz]_p \leq \max([-vx]_p,[-vy]_p)
\]
for all $v \in Q$. Thus, we must show that if $p$ is big enough, then we can find a quadratic residue $v \in Q$ such that $[-vz]_p > [-vx]_p$ and $[-vz]_p > [-vy]_p$.

The simplest case is if $-z$ is not a quadratic residue, for then we can choose $v$ such that $[-vz]_p = p-1$, since $p\equiv 3 \pmod{4}$ implies that $-1$ is not a quadratic residue mod $p$. Since $z \neq x,y$ then this implies that necessarily $p-1 > [-vx]_p$ and $p-1 > [-vy]_p$, as required. Observe that this does not require $p$ to be large.

It remains to consider the case when $-z \in Q$. After replacing $-vz$ by $v$, we see that we are reduced to finding $v \in Q$ such that $[v]_p > [vx/z]_p$ and $[v]_p > [vy/z]_p$. That is, we must find $v \in W_p(x/z) \cap W_p(y/z)$. Proposition \ref{p:intersection} supplies a value of $N$ which ensures that this is possible whenever $p > N$. This concludes the proof.
\end{proof}

\begin{rmk}
Let notation be as in Theorem \ref{t:specialcase} and write $D = D(x/p,y/p;z/p)$. If $p = 2q+1$ is prime with $q$ an odd prime then it's not hard to show that
\begin{enumerate}
\item $D = 0$ if and only if $c < a$ and $c < b$;
\item $D= 1/2q$ or $D = 1/2$ if and only if $a < c$ and $b < c$;
\item $D = 1/q$ or $D = (q+1)/(2q)$ if and only if $a < c < b$;
\item $D = 1$ never occurs.
\end{enumerate}
For such primes, computations suggest that $D \leq 1/q$ whenever $p \geq 59$. The bound of $p = 59$ above cannot be improved upon. For example,
\begin{align*}
  W_{47}(29) &= \{18,25,28,36\}, &  W_{47}(43) &= \{21,32,34,42\}, 
\end{align*}
and hence $W_{47}(29) \cap W_{47}(43) = \emptyset$. This can be used to produce examples of hypergeometric series with lots of bounded primes: unravelling the proof of Theorem \ref{t:specialcase} produces corresponding examples such as
\[
  D\left(\frac{4}{47},\frac{18}{47};\frac{46}{47}\right) = \frac{1}{2}. 
\]
Note that this density is indeed larger than the bound $1/q= 1/23$ from Theorem \ref{t:specialcase}. An effective version of Theorem \ref{t:specialcase} follows for primes $p\equiv 3\pmod{8}$ by Proposition \ref{p:effectiveintersection}, but an effective version for primes $p \equiv 7\pmod{8}$ remains an open problem.
\end{rmk}
\bibliographystyle{plain}

\end{document}